\documentclass{amsart}

\usepackage{amsmath,amsfonts,amsthm,amssymb}
\usepackage[alphabetic]{amsrefs}
\usepackage[OT4]{fontenc}

\usepackage[utf8]{inputenc}
\usepackage[polish,english]{babel}
\usepackage{graphicx}
\DeclareGraphicsExtensions{eps, EPS}
\usepackage{epstopdf}

\newtheorem{theorem}{Theorem}[section]

\newtheorem{lem}[theorem]{Lemma}
\newtheorem{theo}[theorem]{Theorem}

\newtheorem{prop}[theorem]{Proposition}

\theoremstyle{definition}

\newtheorem{defi}[theorem]{Definition}

\theoremstyle{definition}

\theoremstyle{remark}

\numberwithin{equation}{section}

\begin{document}

\title{Nonplanar isoperimetric inequality for random groups}

\author{Tomasz Odrzyg{\'o}{\'z}d{\'z}}
\address{Institute of Mathematics, Polish Academy of Science,
Warsaw, {\'S}niadeckich 8}
\email{tomaszo@impan.pl}

\maketitle

\section{Introduction}

In \cite{gro93} Gromov introduced  the notion of a random finitely presented group on $m \geq 2$ generators at density $d \in (0,1)$. The idea was to fix a set of $m$ generators and consider presentations with $(2m-1)^{dl}$ relators, each of which is a random reduced word of length $l$. Gromov investigated the properties of random groups when $l$ goes to infinity. We say that a property occurs in the Gromov density model \textit{with overwhelming probability (w.o.p.)} if the probability that a random group has this property converges to $1$ when $l \rightarrow \infty$.

There are many important properties of this model: for densities $> \frac{1}{2}$ a random group is trivial w.o.p. \cite[Theorem 11]{inv}; for densities $< \frac{1}{2}$ a random group is, w.o.p., infinite, hyperbolic and torsion-free \cite[Theorem 11]{inv}; for densities $<\frac{1}{5}$ a random group does not have Property (T) with overwhelming probability \cite[Corollary 7.5]{ow11}.

One of the basic tools to investigate the geometry of the Cayley complex of a random group is the ``isoperimetric inequality'' proved in \cite{inv}:

\begin{theo}\label{thm:ie} For any $\varepsilon > 0$, in the Gromov random group model at density $d < \frac{1}{2}$ with overwhelming probability all reduced van Kampen diagrams associated to the group presentation satisfy

$$|\partial D| \geq l(1 - 2d -  \varepsilon) |D|,$$
\end{theo}

Here $\partial D$ denotes to set of boundary edges of diagram $D$ and $|D|$ denotes the number of faces of $D$.

One of the corollaries of this theorem is the fact that in the Gromov density model for densities $<\frac{1}{2}$ a random group is w.o.p. hyperbolic.

The goal of this note is to generalize Theorem \ref{thm:ie} to the class of non-planar diagrams of bounded number of faces. 

\begin{defi}
Suppose $Y$ is a $2$--complex, not necessarily a disc diagram. The
 \emph{cancellation} in $Y$ is
 $$\mathrm{Cancel}(Y) = \sum_{e\in Y^{(1)}} (\deg(e)-1).$$

 Let the \emph{size}  $|Y|$ denote the number of $2$--cells of $Y$.
 \end{defi}
 
By $G$ we will denote the random group with the presentation $\left<S | R \right>$  and by $X$ the Cayley complex of $G$ with respect to this presentation. 

\begin{defi}
 We say that $Y$ is \emph{fulfilled} by a set of relators $R$ if there is a combinatorial map from~$Y$ to the presentation complex $X/G$ that is locally injective around edges (but not necessarily around vertices). 
\end{defi}
 
 In particular, any subcomplex of the Cayley complex $X$ is fulfilled by~$R$. From the definition we see that every $2$-cell $f$ of $Y$ \textit{bears} some relator $r \in R$, which means that the edges of the boundary of $f$ are labeled with the consecutive letters of $r$, there is one edge which corresponds to the first letter of $r$ and there is an orientation of the face $f$, which determines the direction of labeling edges with letters or $r$.

 \begin{defi}
 Let $G = \left<S | R \right>$ be a finite presentation. Let $Y$ be a $2$--complex. Suppose that for $L > 0$ there are chosen $L$ embedded edge paths $A_1, A_2, \dots, A_L \subset Y^{(1)}$  each of which is contained in the boundary of some face $f \in  Y^{(2)}$. The set of \textit{fixed edges} is $\mathrm{Fix}(Y) := \bigcup_{i=1}^L A_i$. Suppose that the edges belonging to $\mathrm{Fix}(Y)$ are prescribed with the elements of $S$ or their inverses. Such $Y$ we call \textit{a diagram with L fixed paths} or shortly \textit{a diagram with L fixed paths} if the value of $L$ is irrelevant.
 \end{defi}
 
 \begin{defi}
Let $Y$ be a diagram with fixed paths. Suppose that the following additional information is given:
\begin{enumerate}
\item which faces of $Y$ bear the same relator
\item for each face there is chosen an edge corresponding to the first letter of the relator labeling this face
\item the orientation of each face.
\end{enumerate}
Such $Y$ we call an \textit{abstract diagram with fixed paths}. We say that $Y$ is  \textit{fulfilled} by a set of relators $R$ if it is fulfilled as $2$-complex, letters on fixed edges agrees with this fulfilling (meaning that under the combinatorial map given by the fulfilling a fixed edge prescribed by $s$ is mapped onto the edge in  $X/G$ also labeled by $s$) and this fulfilling agrees with the additional information.
 \end{defi}
  
 \begin{defi}
A \textit{reduction pair} is a pair of two adjacent faces in $2$-complex $Y$ that under the combinatorial map from~$Y$ to the presentation complex $X/G$ are mapped onto the same $2$-cell.
 \end{defi}
 
Note that if abstract diagram with fixed paths $Y$ is fulfilled and the are no proper powers in the set of relators $R$ fulfilling $Y$, then $Y$ has no reduction pairs. For densities $d < \frac{1}{2}$ with overwhelming probability there are no proper powers in the random set of relators.
 
We will prove the following

\begin{theorem}[local version with fixed edges]\label{lgl}
For each $K, L, \varepsilon >0$, w.o.p. there is no  $2$-complex $Y$ with $|Y| \leq K$ and with at most $L$ fixed paths fulfilling $R$ and satisfying
\begin{equation}
\mathrm{Cancel}(Y) + |\mathrm{Fix}(Y)| >(d+\varepsilon)|Y|l,
\end{equation} where $l$ is the length of the relators in the presentation.
\end{theorem}

Our proof is only a slight modification of the Olliver's proof of \ref{thm:ie}. The crucial point in our reasoning was to define of the $\mathrm{Cancel}(Y)$, which was done thanks to Piotr Przytycki.  We start with

\begin{prop}\label{prop:full}
Let $R$ be a random set of relators at density $d$ and at length $l$. Let $Y$ be a $2$-complex with fixed paths. Then either $\mathrm{Cancel}(Y) + |\mathrm{Fix}(Y)| < (d+2\varepsilon)|Y|l$ or the probability that there exists a tuple of relators in $R$ fulfilling $Y$ is less than $(2m-1)^{- \varepsilon l}$.
\end{prop}

To prove this proposition we need some more definitions. Let $N := |\mathrm{Fix}(Y)|$ and let $n$ be the number of distinct relators in $Y$. For $1 \leq i \leq n$ let $m_i$ be the number of times relator $i$ appears in $Y$. Up to reordering the relators we can suppose that $m_1 \geq m_2 \geq \dots \geq m_n$.

For $1 \leq i_1, i_2 \leq n$ and $1 \leq k_1, k_2 \leq l$ we say that $(i_1, k_1) > (i_2, k_2)$ if $i_1 > i_2$ or $i_1 = i_2$ but $k_1 > k_2$. Suppose that for some $s \geq 2$ an edge $e$ is adjacent to faces: $f_1, f_2, \dots, f_s$ bearing relators $i_1, i_2, \dots, i_s$ accordingly. Suppose moreover that for $1 \leq j \leq s$ the edge $e$ is the $k_j$--th edge of the face $f_j$. Since $Y$ is reduced, for every $1 \leq j, j' \leq s, j \neq j'$ holds: $(i_j, k_j) > (i_{j'}, k_{j'})$ or $(i_j, k_j) < (i_{j'}, k_{j'})$ (otherwise there will be a reduction pair in $Y$). Therefore this relation defines linear lexicographical order, so there is a minimal element $1 \leq j_{min} \leq s$. If $e$ is not a fixed edge we say that edge $e$ \textit{belongs to faces} $f_j$ for $j \in \{1, 2, \dots s\} \setminus \{j_{min}\}$. If $e$ is a fixed edge we say that it \textit{belongs} to all faces to which it is adjacent.

Let $\delta(f)$ be the number of edges belonging to a face $f$. For $1 \leq i \leq n$ let

$$\kappa_i = \max_{f \text{ face bearing relator i}} \delta(f)$$

Note that

\begin{equation}\label{eq:2}
\mathrm{Cancel}(Y) + N = \sum_{f \in Y^{(2)}} \delta(f) \leq  \sum_{1 \leq i \leq n} m_i \kappa_i
\end{equation}

\begin{defi}
 We say that $Y$ is \emph{partially fulfilled} by a set of relators $R$ if there is a combinatorial map from a subcomplex $Y' \subset Y$ to the presentation complex $X/G$ that is locally injective around edges (but not necessarily around vertices).
\end{defi}

\begin{lem}\label{lem:59}
For $1 \leq i \leq n$ let $p_i$ be the probability that $i$ randomly chosen words $w_1, w_2, \dots, w_i$ partially fulfill $Y$ and let $p_0 = 1$. Then
\begin{equation}\label{eq:3}
\frac{p_i}{p_{i-1}} \leq (2m-1)^{-\kappa_i}.
\end{equation}
\end{lem}

\begin{proof}
Suppose that first $i-1$ words $w_1, \dots, w_{i-1}$ partially fulfilling $Y$ are given. We will successively choose the letters of the word $w_i$ in a way to fulfill the complex. Let $k \leq l$ and suppose that the first $k-1$ letters of $w_i$ are chosen. Let $f$ be the face realizing $\delta(f)=\kappa_i$ and let $e$ be the $k$-th edge of the face $f$.

If $e$ belongs to $f$ this means that there is another face $f'$ meeting $e$ which bears relator $i' < i$ or bears $i$ too, but $e$ appears in $f'$ as a $k' < k$-th edge or that $e$ is a fixed edge. In all these cases the letter on the edge $e$ is imposed by some letter already chosen so drawing it at random has probability $\leq \frac{1}{(2m-1)}$.

Combining all these observations we get that the probability to choose at random the correct word $w_i$ is at most $p_{i-1}(2m-1)^{-\kappa_i}$.
\end{proof}

Now we can provide the proof of Proposition \ref{prop:full}.

\begin{proof}[Proof of Proposition \ref{prop:full}]
For $1 \leq i \leq n$ let $P_i$ be the probability that there exists an $i$-tuple of words partially fulfilling $Y$ in the random set of relators $R$. We trivially have:

\begin{equation}\label{eq:4}
P_i \leq |R|^i p_i = (2m-1)^{idl}p_i
\end{equation}

Combining equations (\ref{eq:2}) and (\ref{eq:3}) we get

$$\mathrm{Cancel}(Y) + N \leq  \sum_{i=1}^n m_i(\log_{2m-1}p_{i-1} - \log_{2m-1}p_i)=$$

$$= \sum_{i=1}^{n-1} (m_{i+1} - m_i)\log_{2m-1}p_i -  m_n \log_{2m-1}p_n +  m_1 \log_{2m-1}p_0. $$

Now $p_0 = 1$ so $\log_{2m-1}p_0 = 0$ and we have

$$\mathrm{Cancel}(Y) + N \leq  \sum_{i=1}^{n-1}(m_{i+1} - m_i)\log_{2m-1}p_i -  m_n \log_{2m-1}p_n $$

Now from (\ref{eq:4}) and the fact that $m_{i+1} - m_i \leq 0$ we have

$$\mathrm{Cancel}(Y) + N \leq \sum_{i=1}^{n-1} (m_{i+1} - m_i)(\log_{2m-1}P_i - i d l) -  m_n \log_{2m-1}(P_n - n d l)$$

Observe that $\sum_{i=1}^{n-1} (m_i - m_{i+1})i d l + m_n n d l  = dl \sum_{i=1}^n m_i = dl|Y|$. Hence

$$\mathrm{Cancel}(Y) + N \leq l|Y| d +  \sum_{i=1}^{n-1} (m_{i+1} - m_i)\log_{2m-1}P_i -  m_n \log_{2m-1}P_n$$

Setting $P = \min_i P_i$ and using the fact that $m_{i+1} - m_i \leq 0$ we get

$$\mathrm{Cancel}(Y) + N \leq l|Y|d + (\log_{2m-1}P) \sum_{i=1}^{N-1} (m_{i+1} - m_i) -  m_N \log_{2m-1}P =$$
$$= l|Y|d - m_1 \log_{2m-1}P \leq |Y|(l d - \log_{2m-1}P),$$
since $m_1 \leq |Y|$. It is clear that a complex is fulfillable if it is partially fulfillable for any $i \leq n$ and so:

$$\mathrm{Probability(}\emph{Y is fullfillable by relators of R}\mathrm{)} \leq P \leq (2m-1)^{\frac{|Y|l d - \mathrm{Cancel}(Y) - N}{|Y|}},$$
which was to be proven.
\end{proof}

\begin{proof}[Proof of Theorem \ref{lgl}]
Let $C(K, L, l)$ be the number of abstract complexs with $L$ fixed paths and having at most $K$ faces, each of which is an $l$-gon. It can be easily checked that for fixed $K$, $C(K, L, l)$ grows polynomially with $l$. We know from Proposition \ref{prop:full} that for any reduced abstract complex violating the inequality $\mathrm{Cancel}(Y) + |\mathrm{Fix}(Y)| < (d+\varepsilon)|Y|l$ the probability that it is fulfilled by a random set of relators is $\leq (2m-1)^{ \varepsilon l}$. So the probability that there exists a reduced complex with at most $K$ faces, violating the inequality is $\leq C(K, L, l)(2m-1)^{- \varepsilon l}$, so converges to $0$ when $l \rightarrow \infty$.
\end{proof}

\end{document}